\documentclass[oneside,12pt]{amsart}
\usepackage[T2A]{fontenc}
\usepackage[cp1251]{inputenc}
\usepackage{amssymb}
\usepackage{amsxtra}
\usepackage{amsmath}
\usepackage{amsthm}

\usepackage{hyperref}
\usepackage[all]{xy}

\usepackage[russian,english]{babel}

\DeclareMathOperator{\Hom}{Hom}

\DeclareMathOperator{\SL}{SL}
\DeclareMathOperator{\Sp}{Sp}

\DeclareMathOperator{\ZZ}{{\mathbb Z}}

\newtheorem{lem}{Lemma}[section]
\newtheorem*{thm*}{Theorem}
\newtheorem{thm}[lem]{Theorem}
\newtheorem{prop}[lem]{Proposition}
\newtheorem{cor}[lem]{Corollary}
\newtheorem*{cor*}{Corollary}
\theoremstyle{definition}{  \newtheorem{rem}[lem]{Remark}  }
\theoremstyle{definition}{   }
\theoremstyle{definition}{   }

\begin{document}

\title{Chevalley groups over Laurent polynomial rings}
\author{Anastasia Stavrova}
\thanks{}
\address{St. Petersburg Department of Steklov Mathematical Institute, nab. r. Fontanki 27, 191023 St. Petersburg, Russia}
\email{anastasia.stavrova@gmail.com}

\maketitle

\begin{abstract}
Let $G$ be
a simply connected Chevalley--Demazure group scheme without $\SL_2$-factors. For any unital commutative ring $R$,
we denote by $E(R)$ the standard elementary subgroup of $G(R)$, that is, the subgroup generated by the elementary root
unipotent elements. We prove that the map
$$
G(R[x_1^{\pm 1},\ldots,x_n^{\pm 1}])/E(R[x_1^{\pm 1},\ldots,x_n^{\pm 1}])\to G\bigl(R((x_1))\ldots((x_n))\bigr)/E\bigl(R((x_1))\ldots((x_n))\bigr)
$$
is injective for any
$n\ge 1$, if $R$ is either a Dedekind domain or a Noetherian ring that is geometrically regular over a Dedekind domain
with perfect residue fields. For $n=1$ this map is also an isomorphism.
As a consequence, we show that if $D$ is a PID such that
$SL_2(D)=E_2(D)$ (e.g. $D=\ZZ$), then $G(D[x_1^{\pm 1},\ldots,x_n^{\pm 1}])=E(D[x_1^{\pm 1},\ldots,x_n^{\pm 1}])$.
This extends earlier results for special linear and symplectic groups due to A. A. Suslin and V. I. Kopeiko.
\end{abstract}

\section{Introduction}

For any commutative (unital) ring $R$, let $E_N(R)$ denote the elementary
subgroup of $\SL_N(R)$, i.e. the subgroup generated by elementary matrices $I+te_{ij}$, $1\le i\neq j\le N$, $t\in R$.
A. Suslin~\cite[Corollary 7.10]{Sus} established that for any regular ring $R$ such that $SK_1(R)=1$, one has
$$
\SL_N(R[x_1^{\pm 1},\ldots,x_n^{\pm 1},y_1,\ldots,y_m])=E_N(R[x_1^{\pm 1},\ldots,x_n^{\pm 1},y_1,\ldots,y_m])
$$
for any $N\ge max(3,\dim(R)+2)$ and any $n,m\ge 0$. The corresponding statement for $N=2$ is well known to be wrong:
for any field $k$ one has
$$
\SL_2(k[x_1^{\pm 1},\ldots,x_n^{\pm 1},y_1,\ldots,y_m])\neq E_2(R[x_1^{\pm 1},\ldots,x_n^{\pm 1},y_1,\ldots,y_m]),
$$
as soon as $m\ge 2$ or $m=1$ and $n\ge 1$, see e.g.~\cite{Chu}; to the best of our knowledge, it is not known whether
$\SL_2(k[x_1^{\pm 1},x_2^{\pm 1}])=E_2(k[x_1^{\pm 1},x_2^{\pm 1}])$.

Let $D$ be a principal ideal domain (PID for short). Following~\cite{Lam-book06}, we say that $D$ is a {\it special PID},
if $\SL_n(D)=E_n(D)$ for all $n\ge 2$. (In fact, it is enough to require that $\SL_2(D)=E_2(D)$, see Lemma~\ref{lem:stability} below).
Clearly, a special PID $D$ satisfies $SK_1(D)=1$.
Examples of special PIDs are Euclidean domains and localizations
of 2-dimensional regular local rings at a regular parameter, including the rings $A(x)$ for a discrete valuation ring
$A$~\cite[Corollaries 6.2 and 6.3]{Lam-book06}. Any localization of a special PID is also a special
PID~\cite[Corollary 6.4]{Lam-book06}. The following lemma of V. I. Kopeiko provides one more class of examples.

\begin{lem}\cite[Lemma 4]{Kop-lau2}\label{lem:spPID}
If $D$ is a special PID, then $D((x))$ is a special PID.
\end{lem}
Suslin's theorem cited above implies that, for any special PID $D$, one has
$$
\SL_N(D[x_1^{\pm 1},\ldots,x_n^{\pm 1},y_1,\ldots,y_m])=E_N(D[x_1^{\pm 1},\ldots,x_n^{\pm 1},y_1,\ldots,y_m])
$$
for any $N\ge 3$ and any $n,m\ge 0$. The main result of~\cite{Kop-lau2} is that for any special PID $D$, one has
$$
\Sp_{2N}(D[x_1^{\pm 1},\ldots,x_n^{\pm 1},y_1,\ldots,y_m])=Ep_{2N}(D[x_1^{\pm 1},\ldots,x_n^{\pm 1},y_1,\ldots,y_m])
$$
for any $N\ge 2$ and any $n,m\ge 0$.
Our aim is to extend these two results to all simply connected semisimple Chevalley--Demazure group schemes of isotropic
rank $\ge 2$.

By a Chevalley--Demazure group scheme we mean a split reductive group scheme in the sense of~\cite{SGA3}. These group
schemes are defined over $\ZZ$. Their groups of points are usually called just Chevalley groups.
We say that a Chevalley--Demazure group scheme $G$ has {\it isotropic rank $\ge n$} if and only if every irreducible component of its root system has rank $\ge n$.
For any commutative ring $R$ with 1 and any fixed choice of a pinning, or \'epinglage of $G$ in the sense of~\cite{SGA3},
we denote by $E$ the elementary subgroup functor of $G$. That is, $E(R)$ is the subgroup of $G(R)$ generated
by elementary root unipotent elements $x_\alpha(r)$, $\alpha\in\Phi$, $r\in R$, in the notation of~\cite{Che55,Ma},
where $\Phi$ is the root system of $G$. If $G$ has isotropic rank $\ge 2$, then $E$ is independent of the
choice of the pinning and normal in $G(R)$~\cite{Tad,PS}.
If a simply connected Chevalley--Demazure group scheme $G$ has isotropic rank $1$, it means that $G$ is a direct
product of several simply connected
Chevalley--Demazure group schemes including at least one factor isomorphic to $\SL_2$.
If this is the case, then we always assume that such an isomorphism is fixed throughout every specific
argument, and the corresponding direct factor of $E(R)$ is a fixed
standard elementary subgroup $E_2(R)$ of $\SL_2(R)$; this is to
take care of the fact that $E_2(R)$ is not, in general, a normal subgroup of $\SL_2(R)$.
We also denote
$$
K_1^G(R)=G(R)/E(R);
$$ this is a group if the isotropic rank of $G$ is $\ge 2$,
and a pointed set otherwise.

Our main result is the following theorem.
\begin{thm}\label{thm:sPID}
Let $D$ be a special PID.
Let $G$ be a simply connected Chevalley--Demazure group scheme of isotropic rank $\ge 2$.
Then $K_1^G(D[x_1^{\pm 1},\ldots,x_n^{\pm 1},y_1,\ldots,y_m])=K_1^G(D)=1$ for any $m,n\ge 0$.
\end{thm}

If $n=0$, this result is a special case of our earlier results~\cite[Theorems 1.1 and 1.5]{St-ded}
(see also Theorem~\ref{thm:A1} below). If $n\ge 1$, and if $D$ is semilocal and contains a field, or
if $D$ is itself a field,
the claim of Theorem~\ref{thm:sPID} is already known by~\cite[Corollary 3.3]{St-k1}. In fact, it was even proved there
for arbitrary simply connected semisimple reductive group schemes $G$ of isotropic rank $\ge 2$ and
arbitrary equicharacteristic semilocal regular rings $D$. However, if $n\ge 1$, there is no suitable local-global principle
that would allow to deduce any result for non-semilocal rings from the semilocal case.
Instead, in order to prove Theorem~\ref{thm:sPID} we prove the following result.

\begin{thm}\label{thm:ALau-inj}
Let $A$ be a Dedekind ring, or a Noetherian ring which is geometrically regular over a Dedekind ring $D$ with perfect residue fields.
Let $G$ be a simply connected Chevalley--Demazure group scheme of isotropic rank $\ge 2$.
Then the natural map
$$
K_1^G(A[x_1^{\pm 1},\ldots,x_n^{\pm 1}])\to K_1^G\bigl(A((x_1))\ldots ((x_n))\bigr)
$$ is injective for any $n\ge 1$. If $n=1$, this map is an isomorphism.
\end{thm}

\section{Proof of Theorems~\ref{thm:sPID} and~\ref{thm:ALau-inj}}\label{sec:main}

Recall that a pair $(A,I)$, where $A$ is a commutative ring and $I$ is an ideal of $A$, is called a {\it Henselian
pair} if $I$ is contained in the Jacobson radical of $A$ and for any monic polynomial $f\in A[x]$ and any factorization
$\bar f=g_0h_0$, where $\bar f$ is the image of $f$ in $A/I[x]$ and $g_0,h_0$ are two monic polynomials in $A/I[x]$
generating the unit ideal, there exists a factorization $f=gh$ in $A[x]$ with $g,h$ monic and $\bar g=g_0$, $\bar h=h_0$.

\begin{prop}\label{prop:hens}
Let $G$ be a simply connected Chevalley--Demazure group scheme.
Let $A$ be a commutative ring and let $I$ be an ideal of $A$.
\begin{enumerate}
\item If $I$ is contained in the Jacobson radical of $A$, then the natural map $K_1^{G}(A)\to K_1^{G}(A/I)$ is injective.
\item If $(A,I)$ is a Henselian pair, then $K_1^{G}(A)\cong K_1^{G}(A/I)$.
\end{enumerate}
\end{prop}
\begin{proof}
(1) Since $G(A)\to G(A/I)$ is a group homomorphism and $E(A)\to E(A/I)$ is surjective, it is enough to show that
any element $g\in G(A)$ belongs to $E(A)$, once it is mapped to $1$ under $G(A)\to G(A/I)$.
Let $B,B^-$ be a pair of standard opposite Borel subgroups of $G$, let $U_B$, $U_{B^-}$ be their unipotent radicals,
and let $T=B\cap B^-$ be their common maximal torus.
The group scheme $G$ contains an open $\ZZ$-subscheme $\Omega_B=U_B\cdot T\cdot U_{B^-}$, isomorphic to the direct product of schemes
$U_B\times_{\ZZ}T\times_{\ZZ} U_{B^-}$, and this subscheme $\Omega_B$ is a principal open subscheme~\cite{Che55}.
That is, there an element $d\in\ZZ[G]$ such that $g\in G(A)=\Hom_{\ZZ}(\ZZ[G],A)$ belongs to $U_B(A)T(A)U_{B^-}(A)$
if and only if $g(d)\in A^\times$. Since $G$ is simply connected, we have $T(A)\le E(A)$. Then, if
$g\in G(A)$ is mapped to $1\in G(A/I)$, it follows that $g(d)\in A$ is mapped to $(A/I)^\times$. Since $I$ is contained in the
Jacobson radical of $I$, it follows that $g(d)\in A^\times$. Then $g\in\Omega_B(A)$.
Since $G$ is simply connected, we have $T(A)\le E(A)$. Hence $g\in E(A)$.

(2) By (1) the map $K_1^{G}(A)\to K_1^{G}(A/I)$ is injective.
Since $G$ is affine and smooth, the map $G(A)\to G(A/I)$ is surjective~\cite[Th. I.8]{Gruson}. It follows
that $K_1^{G}(A)\to K_1^{G}(A/I)$ is surjective.
\end{proof}

\begin{rem}
An analog of Proposition~\ref{prop:hens} for isotropic reductive groups $G$ was established in~\cite[\S 7]{GSt} under the
additional assumption that $G$ is defined over a semilocal ring $C$ such that $A$ is a $C$-algebra.
The proof for non-split groups is much more complicated.
\end{rem}

The following corollary generalizes~\cite[Lemma 1 and Remark on p. 1112]{Kop-lau1} for $\SL_n$, $n\ge 2$.

\begin{cor}
Let $G$ be a simply connected Chevalley--Demazure group scheme.
Then $K_1^G(A)=K_1^G(A[[x]])$ for any commutative ring $A$.
\end{cor}
\begin{proof}
Follows from Proposition~\ref{prop:hens} since $(A[[x]],xA[[x]])$ is a Henselian pair.
\end{proof}

\begin{thm}\label{thm:surj}
Let $G$ be a simply connected Chevalley--Demazure group scheme.
 Let $A$ be a commutative ring. Then $G(A((x)))=G(A[x^{\pm 1}])E(A[[x]])$.
In particular, $K_1^G(A[x^{\pm 1}])\to K_1^G(A((x)))$ is surjective.
\end{thm}
\begin{proof}
By~\cite[Corollary 4.4]{St-k1} we have $G(A((x)))=G(A[x^{\pm 1}])G(A[[x]])$. By Proposition~\ref{prop:hens} we have
$G(A[[x]])=G(A)E(A[[x]])$. The claim follows.
\end{proof}

\begin{lem}\label{lem:iso}
Let $G$ be a simply connected Chevalley--Demazure group scheme of isotropic rank $\ge 2$.
Let $A$ be a commutative ring such that $K_1^G(A)=K_1^G(A[x])$. Then
$K_1^G(A[x^{\pm 1}])=K_1^G\bigl(A((x))\bigr)$.
\end{lem}
\begin{proof}
By Theorem~\ref{thm:surj} $K_1^{G,B}(A[x^{\pm 1}])\to K_1^{G,B}(A((x)))$ is surjective. By~\cite[Corollary 3.4]{St-serr}
this map is injective.
\end{proof}

Let $\phi:R\to A$ be a homomorphism of commutative rings.
Following~\cite{Swan} we will say that $\phi$ is {\it geometrically regular}, if $\phi$ is flat and
for every prime ideal $p$ of $R$, and every prime ideal $q$ of $A$ lying over $p$, $A_q/pA_q=k(p)\otimes_A A_q$
is a geometrically regular $k(p)$-algebra, i.e. if for any purely inseparable finite field extension
$k'/k(p)$ the ring $k'\otimes_{k(p)}A_q/pA_q=k'\otimes_A A_q$ is regular in the usual sense. We will just
say that $A$ is a {\it geometrically regular $R$-algebra}, if the structure homomorphism $\phi:R\to A$ is clear from context.

The following theorem is a slight extension of the main result of~\cite{St-ded}.

\begin{thm}\label{thm:A1}\cite[Theorems 1.1, 1.5]{St-ded}
Assume that either $A$ is either a Dedekind ring, or $A$ is a Noetherian ring geometrically regular over a Dedekind ring
with perfect residue fields.
Let $G$ be a simply connected Chevalley--Demazure group scheme of isotropic rank $\ge 2$. Then
$K_1^G(A)=K_1^G(A[x_1,\ldots,x_n])$ for any $n\ge 1$.
\end{thm}
\begin{proof}
If $A$ is a Dedekind ring, the claim is contained in~\cite[Theorem 1.1]{St-ded}. Assume that
$A$ is geometrically regular over a Dedekind ring $D$ with perfect residue
fields. Since $A[x_1,\ldots,x_n]$ is also geometrically regular over $D$ for any $n\ge 1$, it is enough to show that
$K_1^G(A)=K_1^G(A[x])$. By the generalized Quillen-Suslin local-global principle
(see~\cite[Theorem 3.1]{Sus},~\cite[Corollary 4.4]{Sus-K-O1},~\cite[Lemma 17]{PS},~\cite[Theorem 5.4]{Stepanov-elem})
in order to show that $K_1^G(A)=K_1^G(A[x])$, it is enough to show that
$K_1^G(A_m)=K_1^G(A_m[x])$ for every maximal ideal $m$ of $A$. By the very definition of a
geometrically regular ring homomorphism given above, every maximal localization
$A_m$ is geometrically regular over the corresponding prime localization $D_p$ of $D$. By Popescu's theorem~\cite{Po90} (see~\cite[Theorem 1.1]{Swan}),
it follows that $A_m$ is a filtered direct limit of smooth $D_p$-algebras.
Since $K_1^G$ commutes with filtered direct limits, we can assume that $A_m$ is actually a localization
of a smooth $D_p$-algebra, and thus essentially smooth over $D_p$. Then $K_1^G(A_m[x])=K_1^G(A_m)$ by~\cite[Theorem 1.5]{St-ded}.
\end{proof}

\begin{lem}\label{lem:A[[x]]}
Let $D$ be a Dedekind ring with perfect residue fields,
and let $A$ be a Noetherian ring which is geometrically regular over $D$.
Then $A[[x]]$ and $A((x))$ are Noetherian rings geometrically regular over $D$.
\end{lem}
\begin{proof}
Since $A$ is Noetherian, the ring $A[[x]]$ is flat over $A$, and hence its localization $A((x))$ is also flat over $A$.
Then both these rings are flat over $D$. Since $A$ is Noetherian and regular, both these rings are also Noetherian
and regular.
Set $B=A[[x]]$ for brevity, and denote by $\phi:D\to A$ the structure morphism of $A$ over $D$.
It remains to check that for every prime ideal $p$ of $D$, and every prime ideal $q$ of $B$ lying over $p$,
$B_q/\phi(p)B_q=k(p)\otimes_B B_q$
is a geometrically regular $k(p)$-algebra. Since $k(p)$ is perfect by assumption, it is enough to know that
$B_q/\phi(p)B_q$ is a regular local ring.
Now if $\phi^{-1}(q\cap A)=p=0$, then $B_q$ contains the field $D_p=k(p)$, and hence $k(p)\otimes_B B_q=B_q$,
which is obviously regular. If $\phi^{-1}(q\cap A)=p=(\pi)$, where $\pi$ is a prime element of $D$, then let
$n$ be a maximal ideal of $A$ containing $\phi(\pi)$,
and let $m=n+xB$ be the corresponding maximal ideal of $B$. Since $k(p)\otimes_A A_n=A_n/\phi(p)A_n$ is regular, $\phi(\pi)$ is a
regular element of $A_n$. Hence $\phi(\pi)$ is a regular element of $B_m=A_n[[x]]$, and hence $B_m/\phi(p)B_m$ is also regular.
Then $B_q/\phi(p)B_q=(B_m/\phi(p)B_m)_{q/\phi(p)B_m}$ is also regular.
This shows that $A[[x]]$ is geometrically regular over $D$.
Since $A((x))$ is a localization of $A[[x]]$, it follows that $A((x))$ is also geometrically regular over $D$.
\end{proof}

\begin{proof}[Proof of Theorem~\ref{thm:ALau-inj}]
By Theorem~\ref{thm:A1} we have
$$
K_1^G(A[x_2,\ldots,x_n][x_1])=K_1^G(A[x_2,\ldots,x_n]).
$$
Now if $n=1$, the claim of the theorem follows from~Lemma~\ref{lem:iso}. The rest of the claim
we prove by induction on $n$.

Since $A[x_2^{\pm 1},\ldots,x_n^{\pm 1}]=A[x_2,\ldots,x_n]_{x_2\ldots x_n}$ is a localization
of $A[x_2,\ldots,x_n]$,
this implies that
$$
K_1^G(A[x_2^{\pm 1},\ldots,x_n^{\pm 1}][x_1])=K_1^G(A[x_2^{\pm 1},\ldots,x_n^{\pm 1}])
$$
by~\cite[Lemma 4.6]{St-poly} (or~\cite[Lemma 4.2]{Abe}).
Hence by Lemma~\ref{lem:iso} the map
\begin{equation}\label{eq:map}
K_1^G(A[x_1^{\pm 1},\ldots,x_n^{\pm 1}])\to K_1^G\bigl(A[x_2^{\pm 1},\ldots,x_n^{\pm 1}]((x_1))\bigr)
\end{equation}
is an isomorphism. Since the map~\eqref{eq:map}
factors through the map
$$
K_1^G(A[x_1^{\pm 1},\ldots,x_n^{\pm 1}])\to K_1^G\bigl(A((x_1))[x_2^{\pm 1},\ldots,x_n^{\pm 1}]\bigr),
$$
the latter map is also injective.

Now if $A$ is a Dedekind ring, then $A[[x_1]]$ is a regular ring of dimension 2, and, since $x_1$ belongs to every maximal
ideal of $A[[x_1]]$, we conclude that $A((x_1))$ is a regular ring of dimension $1$, hence also Dedekind.
If $A$ is Noetherian and geometrically regular over $D$,
by Lemma~\ref{lem:A[[x]]} the ring $A((x_1))$ is also Noetherian and
geometrically regular over $D$. Summing up, the induction assumption applies to $A((x_1))$, and we are done.
\end{proof}

For the proof of Theorem~\ref{thm:sPID} we need the following lemma which follows from
the stability theorems of M. R. Stein and E. B. Plotkin~\cite{Ste78,Plo93}.

\begin{lem}\label{lem:stability}\cite[Lemma 3.1]{St-ded}
Let $R$ be a Noetherian ring of Krull dimension $\le 1$. If $\SL_2(R)=E_2(R)$, then
$G(R)=E(R)$ for any simply connected Chevalley--Demazure group scheme $G$.
\end{lem}

\begin{proof}[Proof of Theorem~\ref{thm:sPID}]
By Theorem~\ref{thm:A1} we have
$$
K_1^G(D[x_1,\ldots,x_n,y_1,\ldots,y_m])=K_1^G(D[x_1,\ldots,x_n]).
$$
Since $D[x_1^{\pm 1},\ldots,x_n^{\pm 1}]=D[x_1,\ldots,x_n]_{x_1\ldots x_n}$ is a localization
of $D[x_1,\ldots,x_n]$,
by~\cite[Lemma 4.6]{St-poly} (or~\cite[Lemma 4.2]{Abe})
this implies that
$$
K_1^G(D[x_1^{\pm 1},\ldots,x_n^{\pm 1}][y_1,\ldots,y_m])=K_1^G(D[x_1^{\pm 1},\ldots,x_n^{\pm 1}]).
$$
By Theorem~\ref{thm:ALau-inj} the map
$$
K_1^G(D[x_1^{\pm 1},\ldots,x_n^{\pm 1}])\to K_1^G\bigl(D((x_1))\ldots((x_n))\bigr)
$$
is injective. By Lemma~\ref{lem:spPID} $A=D((x_1))\ldots((x_n))$ is a special PID. By definition, it means
that $\SL_2(A)=E_2(A)$. Then by Lemma~\ref{lem:stability} we have $K_1^G(A)=1$ for any
simply connected Chevalley--Demazure group scheme of isotropic rank $\ge 2$. This finishes the proof.
\end{proof}

\renewcommand{\refname}{References}


\begin{thebibliography}{MMMM}

\bibitem[A]{Abe} E. Abe, \emph{Whitehead groups of Chevalley groups over polynomial rings}, Comm. Algebra {\bf 11} (1983),
1271--1307.




\bibitem[Che]{Che55}
C.~Chevalley, \emph{Sur certains groupes simples}, Tohoku Math. J. \textbf{7}
  (1955), 14--66.

\bibitem[Chu]{Chu} Huah Chu, \emph{On the $GE_2$ of graded rings}, J. Algebra \textbf{90} (1984),
208--216.








\bibitem[SGA3]{SGA3} M.~Demazure, A.~Grothendieck, \emph{Sch\'emas en groupes}, Lecture Notes in
Mathematics, vol. 151--153, Springer-Verlag, Berlin-Heidelberg-New York, 1970.





\bibitem[GSt]{GSt} P. Gille, A. Stavrova, \emph{$R$-equivalence
 on group schemes and non-stable $K^1$-functors},
\href{http://arxiv.org/abs/2107.01950}{arXiv:2107.01950}, 46 pp.





\bibitem[Gru]{Gruson} L. Gruson, \emph{Une propri\'et\'e des couples hens\'eliens},
 Colloque d'Alg\`ebre Commutative (Rennes, 1972), Exp. No. 10,
13 pp; Publ. S\'em. Math. Univ. Rennes,  1972.






\bibitem[Ko95]{Kop-lau1}
V. I. Kopeiko, \emph{On
the structure of the special linear groups over Laurent polynomial rings}.
(Russian. English summary)  Fundam. Prikl. Mat., 1:4 (1995), 1111--1114.

\bibitem[Ko99]{Kop-lau2} 
V. I. Kopeiko, \emph{Symplectic groups over Laurent polynomial rings and patching diagrams}.
(Russian. English summary) Fundam. Prikl. Mat. 5, No. 3, 943-945 (1999).

\bibitem[Lam]{Lam-book06}
T.~Y. Lam, \emph{Serre's problem on projective modules}, Springer Monographs in
  Mathematics, Springer-Verlag, Berlin, 2006.




\bibitem[Ma]{Ma} H.~Matsumoto, \emph{Sur les sous-groupes
arithm\'etiques des groupes semi-simples d\'eploy\'es},
Ann. Sci.  de l'\'E.N.S. $4^e$ s\'erie, tome 2, n. 1 (1969), 1--62.


\bibitem[PSt]{PS} V. Petrov, A. Stavrova, \emph{Elementary subgroups of isotropic reductive groups},
St. Petersburg Math. J. {\bf 20} (2009), 625--644.

\bibitem[Plo]{Plo93}
E.~B. Plotkin, \emph{Surjective stabilization of the {$K_1$}-functor for some
  exceptional {C}hevalley groups}, Journal of Soviet Mathematics \textbf{64}
  (1993), no.~1, 751--766.



\bibitem[Pop]{Po90} D. Popescu, \emph{Letter to the Editor: General N\'eron desingularization and approximation},
Nagoya Math. J. {\bf 118} (1990), 45--53.






\bibitem[St14]{St-poly} A. Stavrova, \emph{Homotopy invariance of non-stable $K_1$-functors}, J. K-Theory
{\bf 13} (2014), 199--248.


\bibitem[St15]{St-serr} A. Stavrova,
\emph{Non-stable $K_1$-functors of multiloop groups}, Canad. J. Math. {\bf 68} (2016), 150--178.


\bibitem[St20]{St-ded} A. Stavrova,
\emph{Chevalley groups of polynomial rings over Dedekind domains}, {\it J. Group Theory} {\bf 23} (2020), 121--132.

\bibitem[St22]{St-k1}  A. Stavrova,
 \emph{$\mathbb{A}^1$-invariance of non-stable $K_1$-functors in the equicharacteristic case},
Indag. Math. \textbf{33} (2022), no.~2, 322--333.



\bibitem[Ste78]{Ste78}
M.~R. Stein, \emph{Stability theorems for {$K_{1}$}, {$K_{2}$} and related
  functors modeled on {C}hevalley groups}, Japan. J. Math. (N.S.) \textbf{4}
  (1978), no.~1, 77--108.

\bibitem[Ste13]{Stepanov-elem}
A. Stepanov, \emph{Elementary calculus in {C}hevalley groups over rings},
  J. Prime Res. Math. \textbf{9} (2013), 79--95.



\bibitem[Su]{Sus} A. A.~Suslin,
\emph{On the structure of the special linear group over polynomial rings}, Math. USSR Izv. {\bf 11}
(1977), 221--238.

\bibitem[SuKo]{Sus-K-O1} A.A. Suslin, V.I. Kopeiko, \emph{Quadratic modules and the orthogonal group over polynomial rings},
J. of Soviet Math. {\bf 20} (1982), 2665--2691.


\bibitem[Sw]{Swan} R. G. Swan, \emph{N\'eron-Popescu desingularization}, in Algebra and Geometry (Taipei, 1995),
Lect. Alg. Geom. {\bf 2} (1998), 135--198. Int. Press, Cambridge, MA.

\bibitem[Tad]{Tad} G.~Taddei, \emph{Normalit\'e des groupes \'el\'ementaires dans les groupes de
{C}hevalley sur un anneau}, Contemp. Math \textbf{55} (1986), 693--710.




\end{thebibliography}
\end{document}